\documentclass[12pt,reqno]{amsart}
\usepackage{amsmath}
\usepackage{amsfonts}
\usepackage{amssymb}
\usepackage{hyperref}
\newtheorem{theorem}{Theorem}[section]

\theoremstyle{definition}

\makeatletter \theoremstyle{remark}

\numberwithin{equation}{section}

\clearpage
\begin{document}
\title{Yamabe flow on modified Riemann extension}
\author[ Harish D.]{Harish D.}
\address[Harish D.]{Department of Mathematics, Tripura University, Agartala-799022, INDIA}
\email{itsme.harishd@gmail.com,harishd@tripurauniv.in}
\maketitle
\begin{abstract}
In this paper the rate relations of Riemann, conformal, conharmonic and Weyl curvature tensors under Yamabe flow are studied. Modified Riemann extensions under Yamabe flow is discussed. The paper ends with remarks on some standard metrics.
\end{abstract}
\section{Introduction}
The Yamabe conjecture states that given a compact Riemannian manifold $(M, g_0)$, 
there exists a metric $g$ pointwise conformal to $g_0$ with constant scalar 
curvature $R$. In 1984, R.  Schoen (see \cite{sc}) obtained a complete solution to the 
Yamabe conjecture. It is still an unsolved problem to find a natural evolution 
equation which deforms any Riemannian metric conformally to a constant scalar 
curvature metric.  
In dimension $n=2$ the Yamabe flow is equivalent to the Ricci flow (defined by
$\frac{\partial g_{jk}}{\partial t}=R_{jk}$
where $R_{jk}$ stands for the Ricci tensor). 
However in dimension $n > 2$
the Yamabe and Ricci flows do not agree, since the first one preserves the conformal
class of the metric but the Ricci flow does not in general. The fixed points of the Yamabe flow are metrics of constant scalar curvature in the given conformal class. 

Yamabe flow was initially introduced by Richard Hamilton(unpublished) and was later
studied by Bennett Chow\cite{l},Simon Brendle, Fernando C. Marques, and others. By using the Ricci flow and Yamabe flow introduced by R. Hamilton as the tools, one can prove some interesting Myers type results, which state that if a complete Riemannian manifold with bounded curvature satisfies a curvature pinching condition, then this manifold must be compact\cite{k}.
 
The Yamabe flow equation is the evolution equation
$\frac{\partial g}{\partial t}=(S-R)g$
where $R$ and $S$ are the scalar curvature and average scalar curvature respectively. As flow progresses the metric changes and hence
the properties related to it.
The Riemann extensions of Riemannian or non-Riemannian spaces
introduced by Patterson and Walker \cite{pat}, can be constructed
with the help of the Christoffel coefficients $\Gamma^i_{jk}$ of
corresponding Riemann space or with connection coefficients
$\Pi^i_{jk}$ in the case of the space of affine connection. The
idea of this theory is application of the methods of Riemann
geometry for studying the properties of non-Riemannian spaces. Though the Riemann extensions is rich in geometry, here in our discussions, the modified Riemann extensions  fit naturally in to the frame work. Modified Riemann extensions were introduced in \cite{C-E-P-A} and their properties we list briefly in the next section. 

Yamabe flows are non-linear parabolic differential equations which describe the evolution of geometric structures. Inspired by Hamilton's Ricci flow, the field of yamabe flows has developed tremendously with applications in wide variety of areas such as geometry and topology. Generally, on applying Yamabe flow on Modified Riemann Extension, the resulting metric evolved need not be a Modified Riemann Extension. Here in this paper we study the necessary and sufficient condition for the Yamabe flow to be on Modified Riemann Extension. 

\section{Preliminaries}
Let $(M,g)$ be a Riemannian metric. Then Yamabe flow is defined as 
\begin{equation}\label{imp}
\frac{\partial g}{\partial t}=(S-R)g
\end{equation}
where $R$ is the scalar curvature and $S$ is the average scalar curvature functional 
given by, 
\begin{equation}
S=\frac{1}{V(g)^{\frac{n-2}{n}}}.
\end{equation}

Let $M$ be an n-dimensional manifold and $\nabla$ be a torsion-free affine connection of $M$.
The modified Riemann extension of $(M, \nabla)$ is the cotangent bundle $T^*M$ equipped with a metric $\bar{g}$ whose local components given by
\begin{align}\label{example}
\bar{g}_{ij}=-2 \omega_l \Gamma^l_{ij}+c_{ij},\,\,& \bar{g}_{ij^*}=\delta_i^j, \bar{g}_{i^*j}=\delta_i^j \,\,\textrm{and}\,\, \bar{g}_{i^{*}j^{*}}=0, 
\end{align}
where $\Gamma^l_{ij}$ are the connection coefficients of $M$.
The contravariant components are 
\begin{align}
 \bar{g}^{ij}=0,\,\, \bar{g}^{ij^*}=\delta_i^j,\,\, \bar{g}^{i^{*}j}=\delta_i^j  \,\,\textrm{and}\,\, \bar{g}^{i^*j^*}=2 \omega_l \Gamma^l_{ij}-c_{ij},
 \end{align}
for $i,j$ ranging from $1$ to $n$ and $i^*,j^*$ ranging from $n+1$ to $2n$.\\
 Where $\omega_l$ are extended coordinates and $c_{ij}$ is a symmetric $(0,2)$ tensor on $M$.\\
We note following results for the connection coefficients on extended space,\\
$\bar{\Gamma}^k_{ij}=\Gamma^k_{ij}$,\\
$\bar{\Gamma}^{k^{*}}_{ij}=\omega_l R^l_{kji}+\frac{1}{2}(\nabla_i c_{jk}+\nabla_j c_{ik}-\nabla_h c_{ij})$,\\
$\bar{\Gamma}^k_{i^* j}=0, \,\, \bar{\Gamma}^{k^*}_{i^* j}=-\Gamma^i_{jk}, \,\, \bar{\Gamma}^k_{i^* j^*}=0, \,\, \bar{\Gamma}^{k^*}_{i^* j^*}=0$.
\\
The components of the Riemann curvature tensor of the extended space are given by \\
\\
$\bar{R}^i_{jkl}=R^i_{jkl}$,\\
$ \bar{R}^{i^{*}}_{jkl}=\frac{1}{2}[\nabla_j(\nabla_l c_{ki}-\nabla_i c_{kl})-\nabla_k(\nabla_l c_{ji}-\nabla_i c_{jl})-R^m_{jkl}c_{mi}-R^m_{jki}c_{lm}]$\\  $\hspace*{1.4 cm} +\omega_a(\nabla_j R^a_{ilk}-\nabla_k R^a_{ilj})$,\\
$\bar{R}^{i^{*}}_{{j^*}kl}=R^j_{ilk}$,\\
$\bar{R}^{i^{*}}_{j{k^*}l}=-R^k_{ilj}$,\\
and $\bar{R}^{i^{*}}_{jk{l^*}}=R^l_{kji}$.\\
The others are zero. $i^*,j^*, k^*,l^*$ ranges from $n+1$ to $2n$.
We lower the index in the middle position, to get
\begin{equation}
R_{ijkl}=g_{mk}R^m_{ijl}.
\end{equation}
It may be noted by simple calculation that 
$\bar{R}_{i^*jk^*l}=0$.
Further,
$\bar{R}_{ij}=R_{ij}+R_{ji}$,\;$\bar{R}_{i^*j}=0$ and $\bar{R}_{i^*j^*}=0$.
\section{Main results}
\begin{theorem}\label{star}
Let $(M,g)$ be a Riemannian manifold under Yamabe flow. Then,
\begin{equation}
\frac{\partial g^{ij}}{\partial t}=-(S-R)g^{ij}.
\end{equation}
\end{theorem}
\begin{proof}

From $g^{ij}g_{il}=\delta^{j}_{l}$, on differentiating w.r.t 't' we get
\begin{equation}
\frac{\partial g^{ij}}{\partial t}g_{il}+g^{ij}\frac{\partial g_{il}}{\partial t}=0
\end{equation}
\begin{equation}
g^{ij}\frac{\partial g_{il}}{\partial t}=-g_{il}\frac{\partial g^{ij}}{\partial t}
\end{equation}
So,
\begin{equation}
g^{ij}(S-R)g_{il}=-g_{il}\frac{\partial g^{ij}}{\partial t}
\end{equation}
Hence,
\begin{equation}
\frac{\partial g^{ij}}{\partial t}=-(S-R)g^{ij}
\end{equation}
\end{proof}
\begin{theorem}\label{star1}
If $(M,g)$ is a Riemannian manifold under Yamabe flow, then for the Riemannian curvature tensor $R_{ijkl}$, we have $\frac{\partial R_{ijkl}}{\partial t}=(S-R)R_{ijkl}$
\end{theorem}
\begin{proof}
We have $R_{ijkl}=g(R(e_i,e_j,e_k),e_l)$.
Differentiating partially with respect to 't', we get
\begin{equation}
\frac{\partial R_{ijkl}}{\partial t}=(S-R)\frac{\partial }{\partial t}g(R(e_i,e_j)e_k,e_l)
\end{equation}
Under yamabe flow, $\frac{\partial g_{ij}}{\partial t}=(S-R)g_{ij}$
Hence
\begin{equation}
\frac{\partial R_{ijkl}}{\partial t}=(S-R)R_{ijkl}
\end{equation}
\end{proof}
\begin{theorem}\label{star2}
The Ricci tensor is invariant under yamabe flow .
\end{theorem}
\begin{proof}
\begin{equation}
\frac{\partial R_{jk}}{\partial t}=g^{il}\frac{\partial R_{ijkl}}{\partial t}+\frac{\partial g^{il}}{\partial t} R_{ijkl} 
\end{equation}
Using \ref{star} and \ref{star1} we get,
\begin{equation}
\frac{\partial R_{jk}}{\partial t}=g^{il}(S-R)R_{ijkl}-(S-R)g^{il}R_{ijkl}
\end{equation}
Hence $\frac{\partial R_{jk}}{\partial t}=0$, from which we get the required result.
\end{proof}
\begin{theorem}
If R is the scalar curvature tensor for a Riemannian manifold under yamabe flow, then $\frac{\partial R}{\partial t}=-(S-R)R$
\end{theorem}
\begin{proof}
Differentiating the equation $R=g^{jk}R_{jk}$, we get
$\frac{\partial R}{\partial t}=g^{jk}\frac{\partial R_{jk}}{\partial t}+\frac{\partial g^{jk}}{\partial t}R_{jk}$
Using \ref{star} and \ref{star2} we get,
\begin{equation}
\frac{\partial R}{\partial t}=-(S-R)g^{jk}R_{jk}.
\end{equation}
That is,
\begin{equation}
\frac{\partial R}{\partial t}=-(S-R)R
\end{equation}
\end{proof}
\begin{theorem}\label{star3}
The concircular curvature tensor $C_{ijkl}$ under yamabe flow is given by $\frac{\partial C_{ijkl}}{\partial t}=(S-R)C_{ijkl}$
\end{theorem}
\begin{proof}
The concircular curvature tensor is given by 
\begin{equation}\label{1}
C_{ijkl}=R_{ijkl}-\frac{R}{n(n-1)[g_{il}g_{jk}-g_{jl}g_{ik}]}
\end{equation}
Differentiating partially w.r.t 't', we get
\begin{equation}\begin{split}
\end{split}
\end{equation}
\begin{equation}
\frac{\partial C_{ijkl}}{\partial t}=\frac{\partial R_{ijkl}}{\partial t}+\frac{R(S-R)}{n(n-1)}[g_{il}g_{jk}-g_{jl}g_{ik}]-\frac{2R(S-R)}{n(n-1)}[g_{il}g_{jk}-g_{jl}g_{ik}]
\end{equation}
which gives,
\begin{equation}
\frac{\partial C_{ijkl}}{\partial t}=\frac{\partial R_{ijkl}}{\partial t}-\frac{R(S-R)}{n(n-1)}[g_{il}g_{jk}-g_{jl}g_{ik}]
\end{equation}
Using \ref{1} and \ref{star1} we get,
\begin{equation}
\frac{\partial C_{ijkl}}{\partial t}=(S-R)C_{ijkl}
\end{equation}
\end{proof}
\begin{theorem}\label{star4}
Let $L_{ijkl}$ be the conharmonic curvature tensor of a Riemannian manifold $(M,g)$. Then under Yamabe flow, we have 
\begin{equation}
\frac{\partial L_{ijkl}}{\partial t}=(S-R)L_{ijkl}
\end{equation}
\end{theorem}
\begin{proof}
The conharmonic curvature tensor $L_{ijkl}$ is given by 
\begin{equation}\label{2}
L_{ijkl}=R_{ijkl}-\frac{1}{n-2}[g_{jk}R_{il}+g_{il}R_{jk}-g_{ik}R_{jl}-g_{jl}R_{ik}]
\end{equation}
Therefore,
\begin{equation}
\frac{\partial L_{ijkl}}{\partial t}=\frac{\partial R_{ijkl}}{\partial t}-\frac{1}{n-2}[\frac{\partial g_{jk}}{\partial t}R_{il}+\frac{\partial g_{il}}{\partial t}R_{jk}-\frac{\partial g_{ik}}{\partial t}R_{jl}-\frac{\partial g_{jl}}{\partial t}R_{ik}]
\end{equation}
That is,
\begin{equation}
\frac{\partial L_{ijkl}}{\partial t}=(S-R)R_{ijkl}-\frac{(S-R)}{n-2}[g_{jk}R_{il}+g_{il}R_{jk}-g_{ik}R_{jl}-g_{jl}R_{ik}]
\end{equation}
Hence,
\begin{equation}
\frac{\partial L_{ijkl}}{\partial t}=(S-R)L_{ijkl}
\end{equation}
\end{proof}
Finally for the Weyl curvature tensor, $W_{ijkl}$, we have the following theorem.
\begin{theorem}
Let $W_{ijkl}$ be the Weyl curvature tensor on a Riemannian manifold$(M,g)$. Then under Yamabe flow, we have
\begin{equation}
\frac{\partial W_{ijkl}}{\partial t}=(S-R)W_{ijkl}
\end{equation}
\end{theorem}
\begin{proof}
The Weyl curvature tensor is given by
\begin{equation}
\begin{split}
W_{ijkl}&=R_{ijkl}-\frac{1}{n-2}(g_{jk}R_{il}-g_{ik}R_{jl}+g_{il}R_{jk}-g_{jl}R_{ik})\\&+\frac{R}{(n-1)(n-2)}(g_{il}g_{jk}-g_{jl}g_{ik})
\end{split}
\end{equation}
Using \ref{1} and \ref{2}, we get
\begin{equation}
W_{ijkl}-L_{ijkl}=-\frac{n}{n-2}(C_{ijkl}-R_{ijkl})
\end{equation}
Differentiating w.r.t t,
\begin{equation}
\frac{\partial W_{ijkl}}{\partial t}-\frac{\partial L_{ijkl}}{\partial t}=-\frac{n}{n-2}(\frac{\partial C_{ijkl}}{\partial t}-\frac{\partial R_{ijkl}}{\partial t})
\end{equation}
That is,
\begin{equation}
\frac{\partial W_{ijkl}}{\partial t}=\frac{\partial L_{ijkl}}{\partial t}-\frac{n}{n-2}(\frac{\partial C_{ijkl}}{\partial t}-\frac{\partial R_{ijkl}}{\partial t})
\end{equation}
Using \ref{star1},\ref{star3} and \ref{star4} we get 
\begin{equation}
\frac{\partial W_{ijkl}}{\partial t}=(S-R)L_{ijkl}-\frac{n(S-R)}{n-2}(C_{ijkl}-R_{ijkl})
\end{equation}
That is,
\begin{equation}
\frac{\partial W_{ijkl}}{\partial t}=(S-R)W_{ijkl}
\end{equation}
\end{proof}
\section*{On modified Riemann extension}
Here we study the class of Riemannian metrics of modified Riemann extension. We shall find the condition for which this class satisfies the Yamabe flow equation.
Since the metrics of modified Riemann extension are of constant scalar curvature, we have 
\begin{equation}
\frac{\partial \bar{g_{jk}}}{\partial t}=0.
\end{equation}
Hence we have the following theorem.
\begin{theorem}
Yamabe flow on modified Riemann extension is stationary.
\end{theorem}
This is similar to results with Ricci flow and Dissipative Hypergeometric flow\cite{m}. 
\section*{On some standard metrics}
\begin{theorem}
Let (M,g(t)) be a family of Einstein manifolds. That is $R_{jk}(t)=\lambda(t) g_{jk}(t)$. Then $g(t)$ is a solution to Yamabe flow iff $\lambda(t)=\lambda_0 e^{-\int (S-R)dt}$, where $\lambda_0$ is the initial value of $\lambda(t)$
\end{theorem}
\begin{proof}
Given, $R_{jk}(t)=\lambda(t) g_{jk}(t)$.
So, 
\begin{equation}
\frac{\partial R_{jk}}{\partial t}=\lambda(t)\frac{\partial g_{jk}}{\partial t}+\frac{d\lambda}{dt}g_{jk}
\end{equation}
Using \ref{imp} and \ref{star2} we get,
\begin{equation}
\frac{d\lambda}{dt}g_{jk}=-\lambda(t)(S-R)
\end{equation}
That is,
\begin{equation}
\lambda=\lambda_0 e^{-\int (S-R)dt}
\end{equation}
\end{proof}
Similar result holds for the class of metrics of constant curvature.
\begin{theorem}
Let (M,g(t)) be a family of Riemannian manifolds of constant curvature. That is $R_{ijkl}(t)=\lambda(t)(g_{ik}(t)g_{jl}(t)-g_{il}(t)g_{jk}(t))$. Then $g(t)$ is a solution to Yamabe flow iff $\lambda(t)=\lambda_0 e^{-\int (S-R)dt}$, where $\lambda_0$ is the initial value of $\lambda(t)$
\end{theorem}
\begin{proof}
We have, 
\begin{equation}
R_{ijkl}(t)=\lambda(t)(g_{ik}(t)g_{jl}(t)-g_{il}(t)g_{jk}(t))
\end{equation}
So,
\begin{equation}
\frac{\partial R_{ijkl}}{\partial t}=\lambda(t)(\frac{\partial g_{ik}}{\partial t}g_{jl}+g_{ik}\frac{\partial g_{jl}}{\partial t}-\frac{\partial g_{il}}{\partial t}g_{jk}-g_{il}\frac{\partial g_{jk}}{\partial t})+\frac{d\lambda}{dt}(g_{ik}(t)g_{jl}(t)-g_{il}(t)g_{jk}(t)
\end{equation}
That is,
\begin{equation}
\frac{\partial R_{ijkl}}{\partial t}=2(S-R)R_{ijkl}+\frac{d\lambda}{dt}(g_{ik}(t)g_{jl}(t)-g_{il}(t)g_{jk}(t)
\end{equation}
Using \ref{star1} we get,
\begin{equation}
(S-R)R_{ijkl}=2(S-R)R_{ijkl}+\frac{d\lambda}{dt}(g_{ik}(t)g_{jl}(t)-g_{il}(t)g_{jk}(t)
\end{equation}
That is,
\begin{equation}
\frac{d\lambda}{dt}(g_{ik}(t)g_{jl}(t)-g_{il}(t)g_{jk}(t))=-(S-R)R_{ijkl}
\end{equation}
which gives,
\begin{equation}
\frac{d\lambda}{dt}=-(S-R)\lambda
\end{equation}
which on integrating we get the required result.
\end{proof}

\section{Conclusion}

Thus we have obtained the conditions necessary for a modified Riemann extension under Yamabe flow, to remain as a modified Riemann extension. While describing the flow it may be summarized that for $g^{jk}$ and $R$ we have a negative sign. Otherwise the results of all the theorems are synchronized.

\end{document}